\documentclass[11pt]{amsart}
\usepackage{amsmath,amssymb}
\newtheorem{theorem}{Theorem}

\newtheorem{corollary}{Corollary}
\newtheorem{lemma}{Lemma}

\theoremstyle{definition}

\newtheorem*{acknowledgement*}{Acknowledgement}

\usepackage{latexsym,amssymb}
\begin{document}

\baselineskip=14.5pt
\title{On fractionally dense sets}

\author{ Jaitra Chattopadhyay}
\author{Bidisha Roy}
\author{Subha Sarkar}
\address[Jaitra Chattopadhyay, Bidisha Roy and Subha Sarkar]{Harish-Chandra Research Institute, HBNI, Jhunsi,
Allahabad, India}
\email{jaitrachattopadhyay@hri.res.in}
\email[Bidisha Roy]{bidisharoy@hri.res.in}
\email[Subha Sarkar]{subhasarkar@hri.res.in}

\begin{abstract}
    In this article, we prove some subsets of the set of natural numbers $\mathbb{N}$ and any non-zero ideals of an order of imaginary quadratic fields are fractionally dense in $\mathbb{R}_{>0}$ and $\mathbb{C}$ respectively.     \end{abstract}

     
     \maketitle

\section{Introduction}
It is a basic fact that the set of all rational numbers $\mathbb{Q}$ is dense in the set of all real numbers $\mathbb{R}$.  First, we reformulate this fact in a different way. 

\smallskip

For any subset $A$ of the set of all integers $\mathbb{Z}$ (respectively, the set of all natural numbers $\mathbb{N}$), we define $R(A)$ to be the set of all rational numbers $\frac{a}{a'}$ such that both $a$ and $a'$ lie in $A$ and we call the subset $R(A)$ to be the {\it quotient set } of $A$. 
 
\smallskip
 
 If $A \subset \mathbb{Z}$ (respectively, $\mathbb{N}$) and $R(A)$ is dense in $\mathbb{R}$ (respectively, $\mathbb{R}_{>0}$), then we say $A$ is {\it fractionally dense}  in $\mathbb{R}$ (respectively, $\mathbb{R}_{>0}$).   In this formulation, for example, we  can say that $\mathbb{Z}$ is fractionally dense in $\mathbb{R}$.
 
\smallskip 

The major open problem is to characterize all the subsets of $\mathbb{Z}$ (respectively, $\mathbb{N}$) which are fractionally dense in $\mathbb{R}$ (respectively, $\mathbb{R}_{>0}$). In this direction, many results have already been obtained in  \cite{gems}, \cite{Erdos}, \cite{Toth-Salat}, \cite{Bukor-Toth}, \cite{Bukor-Toth 2} \cite{Dio}, \cite{light}, \cite{hs}, \cite{Salat}, \cite{Salat2} \cite{ps}, \cite{strauch}, \cite{Toth} and \cite{Toth2}. This problem has also been considered in the $p$-adic set-up in \cite{Luca1}, \cite{Luca2}, \cite{Sanna2} and \cite{Sanna}.

\smallskip

Indeed, the most interesting  set, namely,  the set of all prime numbers $\mathbb{P}$ is proved to be fractionally dense in $\mathbb{R}_{>0}$ in \cite{hs} and \cite{Salat}. In \cite{ps}, it is proved that  the set of all prime numbers in a given arithmetic progression is also fractionally dense in $\mathbb{R}_{>0}$. In this article, along with the other results, we generalize this fact.

\smallskip

In \cite{gems}, it is proved that for a given natural number $b\geq 2$, the set of all natural numbers whose base $b$ representation begins with the digit $1$ is fractionally dense if and only if $b = 2, 3, 4$. In the following theorem, we shall generalize this result. The second part along with its converse has been known and it is proved  in \cite{Toth}; however, our method is different and hence we include it here.

\begin{theorem}\label{Th1}
Let $b\geq 2$ be a given integer and let $a$ and $c$ be integers satisfying $1\leq a < c\leq  b$. Consider the subset 
$$
A = \bigcup_{k=0}^{\infty} [ab^k,cb^k) \cap \mathbb{N}
$$
of $\mathbb{N}.$   Then, the following statements are true:
\begin{enumerate}
\item If $ab < c^2$, then the set $B=A \cup \{b^k : k=0,1,2,\ldots\}$ is fractionally dense in $\mathbb{R}_{>0}$. 
\item If  $A$ is fractionally dense in $\mathbb{R}_{>0}$, then $a^2b \leq c^2$. 
\end{enumerate}
\end{theorem}

When we put $a = 1$ (respectively, $a =1$ and $c = 2$) in Theorem \ref{Th1}, we recover the earlier results proved in \cite{gems}. Also, an easy corollary is as follows.

\begin{corollary}\label{Cor11}
Let $b\geq 2$ be a given integer and let $a$ be an integer satisfying $2\leq a\leq b-1$. Then the set of all integers whose base $b$ representation begins with the digit $a$ together with $b^k$ for all $k = 0, 1, \ldots$ is fractionally dense in $\mathbb{R}_{>0}$, if $a= b-1$ or $b-2$. 
\end{corollary}

For $A \subset \mathbb{N}$ and for every $x >1$, we define $A(x)=\{a \in A : a \leq x\}$. We say that $A$ has a {\it natural density} $d(A)$, if  
$$
d(A)=\displaystyle \lim_{n \to \infty} \frac{|A(n)|}{n},
$$ 
provided the limit exists. A subset $A \subset \mathbb{N}$ is said to have {\it lower natural density} $\underline{d}(A)$, if 
$$
\underline{d}(A) = \liminf_{n\to\infty} \frac{|A(n)|}{n}.$$
In \cite{gems} and \cite{Toth}, it was proved that if a subset $A \subset {\mathbb{N}}$ satisfies $\underline{d}(A) \geq \frac{1}{2}$, then $A$ is fractionally dense in $\mathbb{R}_{>0}$. Some related questions have been addressed in \cite{Salat3}. In the following theorem, we consider those subsets $A$ which satisfy $d(A) >0$. Here we note that  Theorem \ref{Th2}   was first proved in \cite{Salat3} and our proof is different.

\begin{theorem}\label{Th2}
Let $U$ and $V$ be given subsets of $\mathbb{N}$  such that $d(U)$  exists and equals $\gamma > 0$. Then 
$\displaystyle \frac{U}{V}=\left\{\frac{u}{v} : u \in U, v \in V \right\}$ is dense in $\mathbb{R}_{>0}$ if and only if $V$ is infinite.
\end{theorem}

Note that if $A\subset \mathbb{N}$ with $d(A) >0$, then $A$ must be an infinite set. Here we give an alternative proof of the following corollary, which was first proved by \v{S}al\'{a}t in \cite{Salat} and again in \cite{light}, by taking $V=U$ in Theorem \ref{Th2}.

\begin{corollary}\label{cor-th2}
Let $U$ be a given subset of $\mathbb{N}$ such that $d(U)$ exists and is positive. Then $U$ is fractionally dense in $\mathbb{R}_{>0}$.
\end{corollary}

Now, we define  the {\it relative density} of a subset $A$ of the set of all prime numbers $\mathbb{P}$ as follows. A subset $A$ of $\mathbb{P}$ has {\it relative density} $\delta(A)$, if
$$
\delta(A) = \lim_{x\to\infty}\frac{|A(x)\cap\mathbb{P}|}{\pi(x)},
$$
provided the limit exists. Here $\pi(x)$ denotes the number of prime numbers $p$ with $p\leq x$.  It readily follows from the definition that if $\delta(A) >0$, then $A$ must be an infinite subset of $\mathbb{P}$. 

\smallskip

Note that if $1\leq a < m$ are integers such that gcd$(a, m) = 1$, then the set $D(a,m)$ of all prime numbers $p$ with $p\equiv a\pmod{m}$ has relative density $\delta(D(a,m)) = 1/\phi(m)$, by {\it Dirichlet's Prime Number Theorem}. Motivated by many examples of subsets of $\mathbb{P}$, we have the following general theorem.

\begin{theorem}\label{Th3}
Let $U \subseteq \mathbb{P}$ and $V \subseteq \mathbb{N}$ be such that $\delta(U)$  exists and equals $\gamma > 0$. Then 
$\displaystyle \frac{U}{V}=\left\{\frac{u}{v} : u \in U, v \in V \right\}$ is dense in $\mathbb{R}_{>0}$ if and only if $V$ is infinite.
\end{theorem}

By taking $V =U$ in Theorem \ref{Th3},  we have the following corollary.

\begin{corollary}\label{cor-th3}
Let $U$ be a given subset of $\mathbb{P}$ such that $\delta(U)$ exists and equals $\gamma >0$. Then $U$ is fractionally dense in $\mathbb{R}_{>0}$. 
\end{corollary}

The following theorem is first proved in \cite{ps}. This can be seen as a corollary to the Theorem \ref{Th3}. However, we give an alternative proof, using the distribution of prime numbers in some special intervals.

\begin{theorem}\label{Th7}
Let  $a,b,m$ and $n$ be given natural numbers  with $m\geq 2$ and $n\geq 2$ such that gcd$(a,m)$ = gcd$(b,n)=1$. Then, the set 
$$\left\lbrace \frac{p}{q} \quad : \quad  p,q \in \mathbb{P},\quad p \equiv a\pmod{m}, \quad q \equiv b\pmod{n} \right\rbrace $$ 
is dense in $\mathbb{R}_{>0}$.
\end{theorem}

In the literature, there is a natural generalization of this concept to the set of all complex numbers $\mathbb{C}$ as follows. Let $K$ be an algebraic number field. Suppose $K$ is not a subfield of $\mathbb{R}$ and $\mathcal{O}_K$ is its ring of integers. A subset $A$ of $\mathcal{O}_K$ is said to be {\it fractionally dense } in $\mathbb{C}$, if its quotient set $R(A)$ is dense in $\mathbb{C}$.

\smallskip
 
When $K=\mathbb{Q}(i)$ with $i = \sqrt{-1}$, Garcia proved in \cite{dense-Gauss} that the set of all prime elements in $\mathcal{O}_{K}=\mathbb{Z}[i]$ is fractionally dense in $\mathbb{C}$. This has been generalized to arbitrary number fields by Sittinger in \cite{Sittinger}.

\smallskip

In \cite{gems}, \cite{Toth-Salat} and \cite{Toth},  it has been proved that if $\mathbb{N}$ has a two-partition, then at least one of them will be fractionally dense in $\mathbb{R}_{>0}$. But, we observe that $\mathbb{Z}$ has a two-partition like $\mathbb{Z} = \mathbb{N} \cup (\mathbb{Z} \setminus \mathbb{N})$ with the fact that none of them is fractionally dense in $\mathbb{R}$. Since $\mathbb{Z}[\sqrt{-d}]$ is a discrete subset of $\mathbb{C}$, it is quite natural to ask the same kind of questions for some particular type of subsets of $\mathbb{Z}[\sqrt{-d}]$. 
In this paper, we shall study the non-zero ideals of the order $\mathbb{Z}[\sqrt{-d}]$ of imaginary quadratic fields.

\smallskip

More precisely, we prove the following theorem, which is a generalization of a result in \cite{gems}.

\begin{theorem}\label{Th4}
Let $d>0$ be a squarefree integer and let $\mathfrak{a}$ be a non-zero ideal in $\mathbb{Z}[\sqrt{-d}]$.  
Let  $\mathfrak{a} =  C \cup D$ with $C \cap D = \emptyset$ be a given two-partition of $\mathfrak{a}$. 
Then, either $C$ or $D$ is fractionally dense in $\mathbb{C}$.
\end{theorem}

Indeed, the result in Theorem \ref{Th4} is optimal in the following sense.

\begin{theorem}\label{Th5}
Let $K$ be an algebraic number field not entirely contained in $\mathbb{R}$ with $\mathcal{O}_K$ its ring of integers.
Let $\mathfrak{a}$ be a non-empty subset in $\mathcal{O}_K$.  Then,  there exist pairwise disjoint non-empty subsets $A, B$ and  $C$ of  $\mathfrak{a}$ such that none of them is fractionally dense in $\mathbb{C}$ and  $\mathfrak{a} = A \cup B \cup C$.
\end{theorem}

\noindent{\bf Remark:} The method we adapt to prove the above Theorem goes through not only for an algebraic number field, but also for $\mathbb{C}$ in general. More precisely, one can prove the following statement. {\it There exist three disjoint subsets $A$, $B$ and $C$ such that $\mathbb{C} = A\cup B\cup C$ and none of them are fractionally dense in $\mathbb{C}$}. 

\smallskip

The next theorem exhibits an infinite subset of prime elements in $\mathbb{Z}[\sqrt{-d}]$ which is not fractionally dense in $\mathbb{C}$. For that we assume additionally that $\mathbb{Z}[\sqrt{-d}]$ is a principal ideal domain. 

\begin{theorem}\label{Th6}
Let $d = 1$ or $2$. Then there exists an infinite set $A$ of prime elements in $\mathbb{Z}[\sqrt{-d}]$ which is not fractionally dense in $\mathbb{C}$.
\end{theorem}

% For alignments use AmS-LaTeX constructions not \eqnarray.

%% - theorems and proofs
%\begin{thm}[optional text]
% The optional material will be typeset as part of the theorem heading
%\end{thm}

%\begin{proof}[Optional proof heading]
% the proof
%\end{proof}
% An end-of-proof symbol (open box) will be typeset at the
% end of the proof.

\section{Preliminaries}

In the preceding section, we have defined the set $D(a,m)$ for natural numbers $a$ and $m$. Note that it is an infinite set if and only if $gcd(a,m) =1$, by {\it Dirichlet's Prime Number Theorem}. For any real number $x > 1$, we let $\pi(a, m, x)$ denote the cardinality of the set of all primes $p\equiv a\pmod{m}$ with $p\leq x$.

\begin{theorem}\label{lem0} (Dirichlet Prime Number Theorem) 
Let $a$ and $m$ be given integers with $a\geq 1$, $m\geq 2$ and $gcd(a,m) =1$. For any real number $x > 1$, we define
$$
G(x) = \frac{\phi(m)\pi(a,m,x)\log x}{x},
$$
where $\phi(m)$ denotes the Euler's phi function.  Then we have 
$$
\lim_{x\to\infty} G(x) = 1.$$
\end{theorem}

The following lemma proves the existence of primes in certain arithmetic progressions in some special intervals. 

\begin{lemma}\label{lem1}
Let  $a$ and $m$ be given natural numbers with $m\geq 2$ and $gcd(a,m) =1$ and let $\alpha > 1$ be a given real number.
Then there exists a positive integer $m_0 = m_0(\alpha)$, depending only on $\alpha$, such that for all integers $n \geq m_0$, we have,  
$$
[\alpha^n,\alpha^{n+1}] \cap D(a,m) \neq \emptyset.
$$ 
\end{lemma} 

\begin{proof}
For all real number $x > 1$, we let 
$$
L(x) = \frac{\log G(x)}{\log \alpha} = \log_\alpha(G(x))
$$ 
where $G(x)$ as defined in Theorem \ref{lem0}. By Theorem \ref{lem0},  we know that $\displaystyle\lim_{x \rightarrow \infty} G(x) = 1$  and hence we 
have $\displaystyle\lim_{x \rightarrow \infty} L(x) = 0$. Therefore, there exists an integer $n_0 >0$ such that 
\begin{equation}\label{eq1}
L(\alpha^{n+1})-L(\alpha^n) >-\frac{1}{2} \mbox{ for every integer } n > n_0.
\end{equation}
Suppose there exists a strictly increasing sequence $\{r_n\}_n$ of natural numbers such that 
$$
[\alpha^{r_n}, \alpha^{r_n+1}]\cap D(a,m) = \emptyset.
$$
First we shall observe the following.  Since $[\alpha^{r_n}, \alpha^{r_n+1}]\cap D(a,m) =\emptyset,$  we get,  $\pi(a,m,\alpha^{r_n}) = \pi(a, m, \alpha^{r_n+1})$.  Hence, we get, 
\begin{equation}\label{eq2}
L(\alpha^{r_n+1})-L(\alpha^{r_n}) = \log_\alpha
\left(\frac{G(\alpha^{r_n+1})}{G(\alpha^{r_n})}\right) = \log_\alpha
\left(\frac{r_n+1}{\alpha r_n}\right) = \epsilon(r_n)-1,
\end{equation}
where $\epsilon(r_n) = \log_\alpha \left(\frac{r_n+1}{r_n}\right)$. Since $\displaystyle\lim_{n\to\infty} r_n = \infty$, there exists $n_1$ such that 
for all $n\geq n_1$, we have $\epsilon(r_n) < 1/2$. 

\smallskip

Put  $m_0 = \max\{n_0,n_1\}$. Then, by \eqref{eq2},  for all $r_n$ with $n\geq m_0$, we get 
$$
L(\alpha^{r_n+1})-L(\alpha^{r_n})= \epsilon(r_n)-1 < -\frac{1}{2},
$$
which is a contradiction to \eqref{eq1}. This proves the lemma.
\end{proof}

We also need the following number field version of the 
Bertrand's postulate  which is due to Hulse and  Ram Murty (see \cite{hr}).

\begin{lemma}\label{lem2}
\noindent{\bf({Bertrand's postulate for number fields})}
Let $K$ be an algebraic number field with $\mathcal{O}_K$ its ring of integers. Then there exists a smallest number $B_K > 1$ such that for every $x > 1$, we can find a prime ideal $\mathfrak{p}$ in $\mathcal{O}_K$ whose norm $N(\mathfrak{p})$ lies inside the interval $[x,B_{K}x]$.
\end{lemma}

\section{Proof of Theorem \ref{Th1}}

Given that $a$, $b$ and $c$ are integers satisfying  $1\leq a < c\leq b$ and the set
$$
A = \bigcup_{k=0}^\infty[ab^k, cb^k)\cap\mathbb{N}.
$$ 

\noindent{(1)} Assume that $ab <c^2$. Then, we prove that  $B = A\cup\{b^k : k =0,1,2, \ldots\}$ is fractionally dense in $\mathbb{R}_{>0}$. 

\smallskip

\noindent{\bf Claim 1.}  $\displaystyle \bigcup_{k \in \mathbb{Z}}\left(\left[\frac{ab^{k}}{c},ab^k\right) \cup \left[ab^k,cb^k\right)\right) = \displaystyle \bigcup_{k\in \mathbb{Z}}\left[\frac{ab^k}{c},cb^k\right) = (0, \infty)$.

\smallskip

The condition $ab < c^2$ implies that any two consecutive intervals of the form $\displaystyle \left[\frac{ab^k}{c},cb^k\right)$ and $\displaystyle \left[\frac{ab^{k+1}}{c},cb^{k+1}\right)$  have non-empty intersection.  Note that, $cb^k \rightarrow \infty$ as $k\to\infty$, and $\frac{ab^k}{c} \rightarrow 0$ as $k\to -\infty$. Therefore, we get
$$
(0,\infty) \subset \displaystyle \bigcup_{k \in \mathbb{Z}}\left[\frac{ab^{k}}{c},ab^k\right) \cup \left[ab^k,cb^k\right)
$$
and hence  Claim 1 follows.

\smallskip

Let $\xi \in \mathbb{R}_{>0}$ be a given element and $\epsilon >0$ be given.  We shall prove that there exists $\alpha \in R(B)$ such that $|\xi - \alpha| <\epsilon$. 

\smallskip

\noindent{\bf Case 1.} $\xi \in [ab^k, cb^k)$ for some integer $k$.

\smallskip

Let $\epsilon > 0$ be given.  Then there exists a sufficiently large natural number $j$  such that $a < b^j \epsilon$. Therefore, we get, 
$$
ab^{j+k} \leq b^j \xi \leq cb^{j+k}.
$$
Then, there exists a non-negative integer $\ell$ satisfying 
\begin{equation}\label{eq3}
ab^{j+k}+\ell \leq b^j \xi \leq a(b^{j+k}+1)+\ell
\end{equation}
with
\begin{equation}\label{eq4}
0 \leq \ell \leq (c-a)b^{j+k}-1.
\end{equation}
Then, by \eqref{eq3} and \eqref{eq4}, we get, 
\begin{eqnarray*}
0 \leq b^j \xi - (ab^{j+k}+\ell) \leq a &\implies & 0 \leq \xi - \frac{ab^{j+k}+\ell}{b^j} \leq \frac{a}{b^j} < \epsilon
\end{eqnarray*}
By \eqref{eq4}, we note that $ab^{j+k}+\ell \geq ab^{j+k}$ and $ab^{j+k}+\ell < cb^{j+k}$ and hence the element $\displaystyle \frac{ab^{j+k}+\ell}{b^j}  = \alpha \in R(B)$, as desired.

\smallskip

\noindent{\bf Case 2.} $\xi \in \displaystyle\left[\frac{ab^k}c, ab^k\right)$ for some integer $k$. 

\smallskip

Since the proof is similar to Case 1, we shall omit the proof here. Hence, we conclude that $B$ is fractionally dense in $\mathbb{R}_{>0}$. This proves the first assertion. 

\smallskip

\noindent{(2)} If possible, suppose $c^2 < a^2b$. We shall show that $A$ is not fractionally dense in $\mathbb{R}_{>0}$.

\smallskip

Let $x,y \in A$ be arbitrary elements. Then, by the definition of $A$,  there exist non-negative integers $k_1$ and  $k_2$ such that 
$$
x \in [ab^{k_1},cb^{k_1}) \mbox{ and } y \in [ab^{k_2},cb^{k_2}).
$$
Therefore, we get
 $$
 \frac{a}{c}b^{k_1 - k_2} < \frac{x}{y} \leq \frac{c}{a}b^{k_1 - k_2}.
 $$
Hence,  every element of $R(A)$ lies in the interval  of the form 
$$
I_\ell=\left(\frac{a}{c}b^\ell,\frac{c}{a}b^\ell\right]
$$
for some $\ell\in \mathbb{Z}$. 

\smallskip

Since by the assumption, $c^2 < a^2b$, we get $\displaystyle \frac{c}{a} < \frac{ab}{c}.$  Therefore, for any integers $j < k$, we have 
$$\frac{c}{a}b^j < \frac{a}{c} b^{j+1} \leq \frac{a}{b}b^k.
$$ 
Thus, we get, $I_j \cap I_k = \emptyset$ for all integers $j$ and $k$ such that $j < k$. Then  the interval $\displaystyle\left(\frac{c}{a}b^j,\frac{a}{c}b^{j+1}\right]$   is non-empty  and is not of the form $I_\ell$ for any $\ell \in \mathbb{Z}$.  Hence, we conclude that 
  $$R(A) \cap \left(\frac{c}{a}b^j,\frac{a}{c}b^{j+1}\right] = \emptyset
  $$ 
  which implies that $A$ is not fractionally dense in $\mathbb{R}_{>0}$. $\hfill\Box$

\section{Proof of Theorem \ref{Th2}}

It is given that $U$ and $V$ are subsets of $\mathbb{N}$ such that $d(U)$ exists and $d(U) = \gamma >0$. 

\smallskip

Suppose  $V$ is an infinite subset of $\mathbb{N}$. For a positive real number $X$, let 
$$
U(X) := \# \{u \in U : u \leq X\}
$$
counts the number of elements of $U$ less than or equal to $X$.  Since $U$ has natural density $\gamma>0$, we have,  
$$
\displaystyle \lim_{X \to \infty} \frac{U(X)}{X} = \gamma > 0  \iff 
U(X)= \gamma X + o(X) \mbox{ for all large enough } X
$$
where $o(X)$ stands for a nonnegative function $g(X)$ such that $g(X)/X \to 0$ as $X\to \infty$. 
Let $a$ and $b$ be two real numbers  satisfying $0 < a < b$. We need to prove that there exist $u \in {U}$ and $v\in V$  such that $a < \displaystyle\frac{u}{v} \leq b$.
 For that, we have
$$
\lim_{X \to \infty} \frac{U(aX)}{U(bX)} = \displaystyle \lim_{X \to \infty} \frac{\gamma aX + o(aX)}{\gamma bX + o(bX)}= \frac{a}{b} < 1.
$$
Put $2\epsilon = \displaystyle 1-\frac{a}{b}$. Since $a < b$, we see that $\epsilon >0$. For this $\epsilon$, there exists $X_0$ such that 
$$
\left|U(aX) - \frac{a}{b}U(bX)\right| < \epsilon U(bX)
$$
holds true for all $X\geq X_0$. This implies 
$$
U(aX)  < \left(\frac{a}b + \epsilon\right) U(bX) < U(bX)
$$ 
for all $X\geq X_0$.  In other words, for all $X\geq X_0$, there exists $u\in U$ such that  $aX < u \leq bX$.

\smallskip

Since $V$ is infinite, we can choose $v \in V$  such that $v\geq X_0$. Therefore, there exists $u\in U$ such that    $av < u \leq bv$ holds true. In other words, we have $a < \displaystyle\frac{u}{v} \leq b$. Hence, we conclude that $\displaystyle\frac{U}{V}$ is dense in $\mathbb{R}_{>0}$.

\smallskip

Conversely, if possible, suppose that $V$ is finite, say, $V=\{v_1,\ldots,v_k\}$. Then
$$
\frac{U}{V}= A_1 \cup \ldots \cup A_k
$$
where $A_j = \displaystyle\left\{\frac{u}{v_j} : u \in U\right\}$ for $j=1,\ldots,k$.

\smallskip

Since $U \subset \mathbb{N}$,  we see that $U$ is a discrete subset of $\mathbb{R}_{>0}$. Hence, each of the sets $A_j$ is discrete and therefore $\frac{U}{V}$, being a finite union of discrete sets, is also discrete. Hence, $\frac{U}{V}$ is not dense in $\mathbb{R}_{>0}$. This proves the assertion.  $\hfill\Box$

\section{Proof of Theorem \ref{Th3}}

It is given that $U$ is a subset of $\mathbb{P}$ such that $\delta(U)$ exists and equals $\gamma >0$. 

\smallskip

Suppose $V$ is an infinite subset of $\mathbb{N}$. For any positive real number $X$, we let 
$$
U(X) = \#\left\{u\in U\ :  \ u\leq X\right\}
$$
which counts the number of element of $U$ less than or equal to $X$. Since $\delta(U)=\gamma > 0$, for all large enough $X$, we have 
$$
U(X) = \gamma\pi(X) + o(\pi(X)).
$$
Therefore, for any real numbers $0<a<b$, we see that 
$$
\lim_{X\to\infty}\frac{U(aX)}{U(bX)} = \frac{a}{b} < 1.$$
The rest of the proof is verbatim to the  proof of Theorem \ref{Th2} and hence we omit the proof here.
$\hfill\Box$

\section{Proof of Theorem \ref{Th7}}

For given natural numbers $a$, $b$, $m$ and $n$ with $m\geq 2$, $n\geq 2$ and gcd$(a,m) = 1 =$ gcd$(b,n)$, let
$$
A = \left\lbrace \frac{p}{q} \quad : \quad  p,q \in \mathbb{P},\quad p \equiv a\pmod{m}, \quad q \equiv b\pmod{n} \right\rbrace
$$ 
be a subset of $\mathbb{R}_{>0}$.
We shall prove that $A$ is dense in $\mathbb{R}_{>0}$. For that, it is enough to show that $A \cap [c,d] \neq \emptyset$ for a non-empty interval $[c,d]$ of $\mathbb{R}_{>0}$. In other words, we prove that $D(a,m)\cap [qc, qd] \ne \emptyset$ 
for some prime number $q\equiv b\pmod{n}$. 

\smallskip

Let $c$, $d$ with $c<d$ be two given positive real numbers. We choose a real number $\alpha$ with $\alpha > 1$ and $\alpha^2 < \displaystyle \frac dc$. 
Then by Lemma \ref{lem1}, there exists an integer $m_0 = m_0(\alpha)$ such that for all integers $k\geq m_0$, we have $D(a, m) \cap [\alpha^k, \alpha^{k+1}] \ne \emptyset.$ Now, we choose a prime $q\in D(b, n)$ such that $q > \displaystyle\frac{\alpha^{m_0}}{c}$. 
Observe that 
$$\log_\alpha (dq)- \log_\alpha (cq) = \log_\alpha \left(\frac{d}{c}\right) > \log_\alpha \alpha^2 = 2.
$$
Thus, there exists an integer $\ell$ such that the interval $[\ell, \ell+1]$  is contained in the interval $[\log_\alpha(cq), \log_\alpha(dq)]$ 
 whence 
 $$[\alpha^\ell,\alpha^{\ell+1}] \subset [cq,dq].
 $$ 
 Since $\alpha^\ell \geq cq > \alpha^{m_0}$, we get  $\ell > m_0$.  Hence,  there exists a prime $p \in D(a,m) \cap
[\alpha^\ell,\alpha^{\ell+1}]$. This proves the theorem. $\hfill\Box$

\smallskip

\section{Proof of Theorem \ref{Th4}}

Let $d > 0$ be a square-free integer and let $\mathfrak{a}$ be a non-zero ideal of $\mathbb{Z}[\sqrt{-d}]$, generated by two elements $a$ and $b$ of $\mathbb{Z}[\sqrt{-d}]$. Let $\mathfrak{a} = C\cup D$ be the given two-partition of $\mathfrak{a}$. Therefore,  we get, $C\cap D = \emptyset$ and $C\cup D = \mathfrak{a}$. Note that, if $C$ is finite, then $D$ is infinite and vice versa. 

\smallskip

\noindent{\bf Case 1.} $C$ is finite.

\smallskip

Let $C = \{\alpha_1, \ldots, \alpha_r\}$ with $\alpha_i \in \mathfrak{a}$. Note that the quotient set of $\mathfrak{a}$ is $$\left\{ \frac{ax+by}{ax'+by'} ~|~ x, y, x',y' \in \mathbb{Z}[\sqrt{-d}] \right\},$$ and we denote it by $R(\mathfrak{a})$. 
Now, we see that 
\begin{equation}\label{th4.1}
 R(\mathfrak{a}) =  R(C\cup D) = R(D) \cup A_1 \cup\ldots \cup A_r,
\end{equation}
where 
$$
A_j = \left\{\frac{\alpha_j}{\beta} \  : \ \beta\in \mathfrak{a}\right\} \cup \left\{\frac{\beta}{\alpha_j} \  : \ \beta\in \mathfrak{a}\right\}
$$ 
for all $j = 1, 2, \ldots, r$. Since $\mathbb{Z}[\sqrt{-d}]$ is discrete in $\mathbb{C}$, we see that $A_j$'s are nowhere dense subsets in $\mathbb{C}$. Since $R(\mathfrak{a})$ is dense in $\mathbb{C}$, we see that $R(D)\cup A_1\cup \ldots\cup A_r$ is dense in $\mathbb{C}$, where $A_1\cup\ldots \cup A_r$ is a nowhere dense subset in $\mathbb{C}$. 

\smallskip

If $R(D)$ is not dense in $\mathbb{C}$, then there exists an open ball $B$ such that $B\cap \overline{R(D)} = \emptyset$.  Therefore, $B \subset \overline{A_1\cup\ldots\cup A_r} = A_1\cup\ldots \cup A_r$ which is a contradiction, as $A_1\cup\ldots\cup A_r$ has empty interior. 
Hence, $D$ is fractionally dense in $\mathbb{C}$.

\smallskip

\noindent{\bf Case 2.} Both the sets $C$ and $D$ are infinite subsets of $\mathfrak{a}$.

\smallskip

Suppose that neither $C$ nor $D$ is fractionally dense in $\mathbb{C}$. Then there exists $\epsilon > 0$ and non-zero complex numbers $ \alpha$ and $\beta$ such that 
\begin{equation}\label{th4.2}
B(\alpha, \epsilon ) \cap R(C) = \emptyset \mbox{ and } B(\beta, \epsilon ) \cap R(D) = \emptyset,
\end{equation}
where $B(z, r)$ denotes the open ball of radius $r$, centered at $z$ in the complex plane. 
 Now, choose a sufficiently large integer $n_0$  satisfying
\begin{equation}\label{th4.3}
\frac{\left(\begin{aligned}& | (1+\sqrt{-d})(a+b)|+|(1+\sqrt{-d})\beta(a+b)|\\
&\qquad + |(1+\sqrt{-d})\alpha\beta( a+b)|\end{aligned}\right)^2}{n_0} < \epsilon
\end{equation}
and
\begin{equation}\label{th4.4}
  \frac{\left(|(1+\sqrt{-d}) \alpha (a+b)| +|(1+\sqrt{-d})\alpha \beta (a+b)|\right)^2 }{n_0}  <  \epsilon.
\end{equation}
Once $n_0$ is chosen, as both $C$ and $D$ are infinite sets, we can find  $\gamma \in C$ satisfying
\begin{equation} \label{th4.5}
 |\gamma|^2 > n_0 | \alpha |^2, \ \  |\gamma|^2 > n_0 | \beta |^2 \ \mbox{ and } |\gamma |^2 > n_0 | \alpha \beta|^2
\end{equation}
together with the following constraint  
\begin{equation}\label{th4.6}
D_1\cap D \neq \emptyset,
\end{equation}
where 
$$
D_1 = \{ \gamma \pm a, \gamma \pm b , \gamma \pm a \pm b\}.
$$ 
To see this fact, suppose, if possible, that for every $\gamma \in C$ satisfying \eqref{th4.5}, we have  
$ D_1 \cap D = \emptyset$. 
This implies that $D$ is bounded. Since $\mathbb{Z}[\sqrt{-d}]$ is discrete,  it follows that $D$ is finite, which is a contradiction. 
Also, note that all the elements of $D_{1}$ can be written as $ \gamma \pm \epsilon a \pm \epsilon ^{'} b$ for some $ \epsilon, \epsilon^{'} \in \{0,1\}$ such that $(\epsilon, \epsilon') \ne (0, 0)$.

\smallskip

Now, write the complex number 
$$\frac{\gamma}{\alpha\beta} = \gamma_1 a + \gamma_2 b, \mbox{ for some } \gamma_1 = x_1 +\sqrt{-d} y_1, \gamma_2 = x_2 +\sqrt{-d} y_2
$$
such that  $x_1, y_1, x_2, y_2 \in \mathbb{R}$ and define 
 $$
    s =(\langle x_1\rangle + \sqrt{-d} \langle y_1 \rangle)a + (\langle x_2 \rangle + \sqrt{-d} \langle y_2 \rangle)b. 
  $$
  where \[
\langle x\rangle = 
\begin{cases}
\lceil x \rceil; & \mbox{ if } x >0 \\
\lfloor x \rfloor; & \mbox{ if } x <0
\end{cases}
\] 
and $\lceil x\rceil$ is the ceiling of $x$ and $\lfloor x\rfloor$ is the floor of $x$. 
  Note that 
  \begin{equation}\label{th4.7}
  s = \frac{\gamma}{\alpha\beta} \pm (\epsilon_1 \pm\sqrt{-d}\epsilon'_1)a \pm (\epsilon_2 \pm\sqrt{-d}\epsilon'_2)b \in \mathfrak{a},
 \end{equation}
  for some $\epsilon_1, \epsilon'_1, \epsilon_2, \epsilon'_2 \in [0, 1)$.

\smallskip

\noindent{\bf Claim 1.} $s \not\in C\cup D$

\smallskip

If we prove the above claim, then we get a contradiction to the fact that  $s\in \mathfrak{a} = C\cup D$. Hence, to finish the proof of this theorem, it is enough to prove the claim. Since $s\in \mathfrak{a} = C\cup D$ and $C\cap D = \emptyset$, the element $s$ lies inside $C$ or $D$ but not both. 
If possible, we assume that $s\in C$.

\smallskip
 
 Now,  we write $ \alpha s =\delta_1 a+ \delta_2 b $  for some $\delta_1=x_3 +\sqrt{-d}y_3$, $\delta_2= x_4+\sqrt{-d} y_4$  such that $ x_3$, $y_3$, $x_4$, $y_4 \in \mathbb{R}$ and define 
\[
    t = (\langle x_3\rangle + \sqrt{-d} \langle y_3 \rangle)a + (\langle x_4 \rangle + \sqrt{-d} \langle y_4 \rangle)b,   
  \]
  where $\langle x \rangle$ is defined as above.
  
  \smallskip
  
  Then 
  \begin{equation}\label{th4.8}
  t = \alpha s \pm (\epsilon_{3} \pm \sqrt{-d}\epsilon'_3)a \pm (\epsilon _4 \pm\sqrt{-d}\epsilon'_4)b \in \mathfrak{a}
  \end{equation}
  for some $ \epsilon_{3}, \epsilon^{'}_3, \epsilon_4, \epsilon'_4 \in [0, 1)$. Let $d(z_1, z_2)$ denote the usual distance function in $\mathbb{C}$ and we estimate the distance between $t/s$ and $\alpha$ as follows. 
  
Since, by \eqref{th4.5}, the inequality $\displaystyle |s|^2 \geq \left|\frac{\gamma}{\alpha\beta}\right|^2 > n_0$  holds, we see that   
  \begin{align*}
  d\left(\frac{t}{s}, \alpha\right)^2 & =  \left| \frac{t-\alpha s}{s}\right|^2 \\
                           & \leq \left| \frac{(\epsilon_{3} \pm \sqrt{-d}\epsilon'_3)a \pm (\epsilon _4 \pm\sqrt{-d}\epsilon'_4)b}{s}\right|^2 \\ 
                           & \leq \left|\frac{(1+\sqrt{-d})(a+b)}{s}\right|^2 
                            <  \epsilon,
  \end{align*}
  by \eqref{th4.3} and for some  $\epsilon_3, \epsilon^{'}_3, \epsilon_4, \epsilon'_4 \in [0, 1)$.    If $t\in C$, then $t/s \in R(C)$. Therefore, by \eqref{th4.2}, we conclude that $t\not\in C$, which implies $t\in D$.
    
  \smallskip

  Now, we calculate the distance between the elements of the form $\delta/t$ for any $\delta \in D_1$ and $\beta$ as follows. Let $\delta \in D_1$ be an arbitrary element and consider
  \begin{align*}
  & d\left(\frac{\delta}{t}, \beta\right)^2   = \frac{\left| \delta - \beta t \right|^2}{|t|^2} 
                           = \frac{\left| \delta-\beta (\alpha s \pm (\epsilon_{3} \pm \sqrt{-d}\epsilon'_3)a \pm (\epsilon _4 \pm\sqrt{-d}\epsilon'_4)b)\right|^2}{|t|^2} \\
                           & \leq  \frac{\left| \begin{aligned} 
                            \delta &\pm \beta((\epsilon_{3} \pm \sqrt{-d}\epsilon'_3)a \pm (\epsilon _4 \pm\sqrt{-d}\epsilon'_4)b)\\
                           &- \alpha \beta \left(\frac{\gamma}{\alpha\beta} \pm (\epsilon_1 \pm\sqrt{-d}\epsilon'_1)a \pm (\epsilon_2 \pm\sqrt{-d}\epsilon'_2)b\right)
                           \end{aligned}\right|^2}{|t|^2}\\
                           & \leq \frac{\left|\begin{aligned}
                            \pm \epsilon &a \pm \epsilon^{'}b+\beta( (\epsilon_{3} \pm \sqrt{-d}\epsilon'_3)a \pm (\epsilon_4 \pm\sqrt{-d}\epsilon'_4)b)\\
                           &+ \alpha\beta(\pm (\epsilon_1 \pm\sqrt{-d}\epsilon'_1)a \pm (\epsilon_2 \pm\sqrt{-d}\epsilon'_2)b)\end{aligned}\right|^2}{|t|^2}\\
                           & <  \epsilon
                             \end{align*}
                             by \eqref{th4.7}, \eqref{th4.8} and using the estimate
$$
      |t|^2 \geq \left|\alpha s\right|^2 \geq \left|\frac{\gamma}{\beta}\right|^2 > n_{0}
      $$  
      together with the inequality \eqref{th4.3}. Note that the above inequality is true for all $\delta \in D_1$. By \eqref{th4.6}, we know that $|D_1\cap D| \geq 1$ and hence there exists a $\delta \in D_1$ such that $\delta \in D$ also. For this $\delta$, we get $\displaystyle\frac{\delta}{t} \in B( \beta , \epsilon )\cap R(D)$, which is a contradiction.  Therefore, we conclude that $s\not\in C$ and hence $s\in D$. 
      \smallskip      
      Again, we write $ \beta s =\delta'_1 a+ \delta'_2 b $ 
for some $\delta'_1=x'_3 +\sqrt{-d}y'_3$, $\delta'_2= x'_4+\sqrt{-d} y'_4$ such that $x'_3, y'_3, x'_4, y'_4 \in \mathbb{R}$ and consider
$$
    t^{'} = (\langle x'_3\rangle + \sqrt{-d} \langle y'_3 \rangle)a + (\langle x'_4 \rangle + \sqrt{-d} \langle y'_4 \rangle)b,   
  $$
  where $\langle x \rangle$ is defined similarly as above.   Hence, 
  $$
  t' = \beta s \pm (\epsilon_{5} \pm \sqrt{-d}\epsilon'_5)a \pm (\epsilon _6 \pm\sqrt{-d}\epsilon'_6)b \in \mathfrak{a}
  $$  
  for some $ \epsilon_{5}, \epsilon^{'}_5, \epsilon_6, \epsilon'_6 \in [0, 1)$ and  we get 
  $$ |t'|^2  \geq |\beta s|^2 \geq \left| \frac{\gamma}{\alpha}\right|^2 > n_{0}
  $$ 
  by \eqref{th4.5}. 
  Again, by the similar arguments,  we can show that 
  $$ d\left(\frac{ t'}{s}, \beta\right)^2 < \epsilon 
  $$ 
  and conclude $t' \in C$ as $B( \beta, \epsilon) \cap R(D) = \emptyset $.
  
  \smallskip
  
  Now, we consider
  \begin{align*}
  & d\left(\frac{\gamma}{t'}, \alpha\right)^2\
   = \frac{\left|\gamma - \alpha(\beta s \pm (\epsilon_{5} \pm \sqrt{-d}\epsilon'_5)a \pm (\epsilon _6 \pm\sqrt{-d}\epsilon'_6)b)\right|^2}{|t'|^2}\\
                                 & = \frac{\left| \gamma - \alpha \beta s + \alpha(\pm (\epsilon_{5} \pm \sqrt{-d}\epsilon'_5)a \pm (\epsilon _6 \pm\sqrt{-d}\epsilon'_6)b)\right|^2}{|t'|^2}\\
                                 & = \frac{\left| \begin{aligned}
                                   \gamma - &\alpha \beta \left(\frac{\gamma}{\alpha\beta} \pm (\epsilon_1 \pm\sqrt{-d}\epsilon'_1)a \pm (\epsilon_2 \pm\sqrt{-d}\epsilon'_2)b\right)\\ 
                                 & + \alpha(\pm (\epsilon_{5} \pm \sqrt{-d}\epsilon'_5)a \pm (\epsilon _6 \pm\sqrt{-d}\epsilon'_6)b) \end{aligned}\right|^2}{|t'|^2} \\
                                 & =\frac{\left|\begin{aligned}  \alpha \beta(\pm & (\epsilon_1 \pm\sqrt{-d}\epsilon'_1)a \pm (\epsilon_2 \pm\sqrt{-d}\epsilon'_2)b)\\
                                 &+\alpha\left(\pm (\epsilon_{5} \pm \sqrt{-d}\epsilon'_5)a \pm (\epsilon _6 \pm\sqrt{-d}\epsilon'_6)b\right)\end{aligned}\right|^2}{|t'|^2}\\
                                 & < \epsilon
\end{align*}
by \eqref{th4.4} and the above estimate.  Thus, we get, 
$$\frac{\gamma}{t'} \in B(\alpha , \epsilon ) \cap R(C),
$$ 
which is a contradiction again.  This proves the Claim 1 and the theorem. 
 $\hfill\Box$

\section{Proof of Theorem \ref{Th5}}

We want to find a three-partition of the set $\mathfrak{a}$  such that none of which is fractionally dense in $\mathbb{C}$. If $\mathfrak{a}$ is finite, then there is nothing to prove. Now if $\mathfrak{a}$ is infinite, let us consider the sets
$$
A = \bigcup_{k=0}^\infty \left\{ a+ib : a+ib \in \mathfrak{a} \text{ and }  a^2 + b^2 \in [ 5^k, 2 \cdot 5^k) \right\},
$$
$$
B = \bigcup_{k=0}^\infty \left\{ a+ib : a+ib \in \mathfrak{a} \text{ and }  a^2 + b^2 \in [ 2 \cdot 5^k, 3 \cdot 5^k) \right\}
$$ 
and
 $$
 C = \bigcup_{k=0}^\infty\left \{ a+ib : a+ib \in \mathfrak{a} \text{ and }  a^2 + b^2 \in [3 \cdot 5^k, 5\cdot 5^k) \right\}.
 $$ 
It is easy to observe that $\mathfrak{a} = A \cup B \cup C \subseteq \mathcal{O}_K$ with $A\cap B=\emptyset$, $B\cap C = \emptyset$ and $C\cap A = \emptyset$.

\smallskip

\noindent{\bf Claim 1.} $C$ is not fractionally dense in $\mathbb{C}$.

\smallskip

First note that if $\frac{p}{q} \in R(C)$, then $\frac{p}{q}$ lies 
in an annulus of the form 
$$B_\ell = \left\{x+iy \in \mathbb{C} : \frac{3}{5} 5^\ell < x^2 + y^2 < \frac{5}{3}5^\ell \right\},
$$ for some integer $\ell$. 

\smallskip

For any integer $j < k$, the following inequality  
$\displaystyle \frac{5}{3}\cdot 5^j < \frac{3}{5} \cdot 5^k$
holds true. Hence  for all integers $  j < k $, we have, $B_{j} \cap B_{k} = \emptyset$. Thus, for any integer $\ell$, the set 
$$M=\{ x+iy  \in \mathbb{C} \ : \  \frac{5}{3}5^\ell< x^2 + y^2 < \frac{3}{5} 5^{\ell+1} \} \neq \emptyset
$$ 
and satisfies $M \cap R(C)=\emptyset$. Thus, $C$ is not fractionally dense in $\mathbb{C}$.

\smallskip

Similarly, one can prove  that neither $A$ nor $B$ is fractionally dense in $\mathbb{C}$. This completes the proof of the theorem. $\hfill\Box$

\section{Proof of Theorem \ref{Th6}}
When $d=1$ or $2$, it is well-known that $\mathbb{Z}[\sqrt{-d}]$ is the ring of integers of $\mathbb{Q}(\sqrt{-d})$ and it is principal ideal domain.
We shall construct an infinite set $A$ of prime elements in $\mathbb{Z}[\sqrt{-d}]$ which is not fractionally dense in $\mathbb{C}$.

\smallskip

By Lemma \ref{lem2}, there exists a smallest number $B > 1$  such that for every real number $x > 1$, one can find a prime ideal $\mathfrak{p}$ of $\mathbb{Z}[\sqrt{-d}]$ whose norm $N(\mathfrak{p}) \in [x, Bx]$. Since $\mathbb{Z}[\sqrt{-d}]$ is a principal ideal domain, every prime ideal $\mathfrak{p}$ is generated by a prime element, say, $\alpha_{\mathfrak{p}}$ and $N(\mathfrak{p}) = N(\alpha_{\mathfrak{p}}).$ Thus, we conclude that for every real number $x > 1$, there exists a prime element $\alpha \in \mathbb{Z}[\sqrt{-d}]$ whose norm $N(\alpha) \in [x, Bx]$. 

\smallskip

In other words, for each natural number $n> 1$, there exists a prime element $\alpha_n \in \mathbb{Z}[\sqrt{-d}]$ whose norm $N(\alpha_n) \in [B^{2n-1}, B^{2n}]$. Let $A$ be the subset of $\mathbb{Z}[\sqrt{-d}]$ which consists precisely of those $\alpha_{n}$'s.  Clearly the set $A$ is infinite. 

\smallskip

\noindent{\bf Claim:}  $A$ is not fractionally dense in $\mathbb{C}$.

\smallskip

Let $1 < m < n$ be any given integers. Then by the above argument, we know that $N(\alpha_m) \in [B^{2m-1}, B^{2m}]$ and $N(\alpha_n) \in [B^{2n-1}, B^{2n}]$. 
Therefore, we get
$$
N(\alpha_m) \leq B^{2m} \leq B^{2(n-1)} < B^{2n-1} \leq N(\alpha_n).
$$
Hence, we get
$$
\frac{N(\alpha_m)}{N(\alpha_n)} < \frac{1}{B} \mbox{ and } \frac{N(\alpha_n)}{N(\alpha_m)} > B.
$$
Thus, if we consider the annulus 
$$
AN= \left\{z \in \mathbb{C}  \ : \  \sqrt{\frac{1}{B}} < |z| < \sqrt{B}\right\},
$$ 
then,  no element of  $R(A)$ lies inside $AN$. 
%Therefore, $A$ cannot be fractionally dense in $\mathbb{C}$. 
$\hfill\Box$

\section{Concluding Remarks} 

\begin{itemize}
\item[1.] 
In \cite{gems}, a necessary and sufficient condition for the set of positive integers whose $b$-ary expansion begins with the digit $1$ to be fractionally dense in $\mathbb{R}_{>0}$ was proved. In Corollary \ref{Cor11}, we have provided a sufficient condition for the set of integers whose $b$-ary expansion begins with the digit   $a$ to be  fractionally dense in $\mathbb{R}_{>0}$. 

\smallskip

A natural question is whether the converse is true in Corollary \ref{Cor11}. More generally, one may ask the truth of the  following statement.

%\begin{quest}
\noindent{\bf Question.} {\it If the set $B$ in Theorem} \ref{Th1} {\it is fractionally dense in $\mathbb{R}_{>0}$, then is $b\leq c^2/a$?} 
%\end{quest}

\item[2.]
Let $K$ be an algebraic number field not contained in $\mathbb{R}$ such that $[K:\mathbb{Q}] \geq 3$. Then it is well known that $\mathcal{O}_K$ is dense in $\mathbb{C}$. Now let $\mathcal{O}_K = A \cup B$ be a partition  of $\mathcal{O}_K$ with at least one of them (say $A$) finite. Then $B$ is dense in $\mathbb{C}$ and hence $B$ is fractionally dense in $\mathbb{C}$. Now if both $A$ and $B$ are infinite, a natural question is whether at least one of $A$ or $B$ is fractionally dense or not.

\smallskip 

More precisely, we may ask the following: 

%\begin{quest}
\noindent{\bf Question.} {\it Let $\mathcal{O}_K = A \cup B$ be a partition of $\mathcal{O}_K$ such that both $A$ and $B$ are infinite, then is at least one of $R(A)$ or $R(B)$ dense in $\mathbb{C}$? If not, then find two infinite sets $A$ and $B$ such that none of them is fractionally dense in $\mathbb{C}$ and $\mathcal{O}_K = A \cup B$.}
%\end{quest} 
\end{itemize}

\bigskip

\noindent{\bf Acknowledgements.} We express our sincere thanks to Professors J. Bukor, C. Sanna and  O. Strauch,  for pointing out some of the important references in the literature. Also, we thank Professor R. Thangadurai for carefully going through the earlier version of this manuscript. We are grateful to the referee for his/her valuable comments that greatly improved the presentation of the paper. We acknowledge our Institute for providing us the necessary and excellent facilities to carry out this research.


\begin{thebibliography}{99}

\bibitem{gems} 
B. Brown, M. Dairyko, S. R. Garcia, B. Lutz and M. Someck, {Four quotient set gems}, {\it Amer. Math. Monthly}, {\bf 121} (2014), 590-599.

\bibitem{Erdos} 
J. Bukor, P. Erd\H{o}s, T. \v{S}al\'{a}t and J. T. T\'{o}th, {Remarks on the $R$-density of sets of numbers. II}, {\it Math. Slovaca}, {\bf 47} (1997), 517-256.

\bibitem{Toth-Salat}
J. Bukor, T. \v{S}al\'{a}t and J. T. T\'{o}th, {Remarks on $R$-density of sets of numbers}, {\it Tatra Mt. Math. Publ.}, {\bf 11} (1997), 159-165.

\bibitem{Bukor-Toth}
J. Bukor and J. T. T\'{o}th, {On accumulation points of ratio sets of positive integers,} {\it Amer. Math. Monthly}, {\bf 103} (1996), 502-504.

\bibitem{Bukor-Toth 2}
J. Bukor and J. T. T\'{o}th,
{On some criteria for the density of the ratio sets of positive integers}, {\it JP Journal of Algebra, Number Theory and Applications}, {\bf 3}(2003), 277-287.

\bibitem{dense-Gauss}
S. R. Garcia, {Quotients of Gaussian primes}, {\it Amer. Math. Monthly}, {\bf 120} (9) (2013), 851-853.

\bibitem{Luca1}
S. R. Garcia, Y. X. Hong, F. Luca, E. Pinsker, C. Sanna, E. Schechter and A. Starr, {$p$-adic quotient sets}, {\it Acta Arith.}, {\bf 179} (2) (2017), 163-184.

\bibitem{Luca2}
S. R. Garcia and F. Luca, {Quotients of Fibonacci numbers}, {\it Amer. Math. Monthly}, {\bf 123} (2016), 1039-1044.

\bibitem{Dio}
S. R. Garcia, V. Selhorst-Jones, D. E. Poore and N. Simon, {Quotient sets and Diophantine equations}, {\it Amer. Math. Monthly}, {\bf 118} (2011), 704-711.

\bibitem{light}
S. Hedman and D. Rose, {Light subsets of $\mathbb{N}$ with dense quotient sets}, {\it Amer. Math. Monthly}, {\bf 116} (2009), 635-641.

\bibitem{hs}
D. Hobby and  D. M. Silberger,  {Quotients of primes}, {\it Amer. Math. Monthly}, {\bf 100} (1) (1993), 50-52.

\bibitem{hr}
T. A. Hulse and M. Ram Murty, {Bertrand's postulates for number fields}, {\it Colloquium Mathematicum}, {\bf 147} (2), (2017) 165-180.

\bibitem{Sanna2}
P. Miska and C. Sanna, {$p$-adic denseness of members of partitions of $\mathbb{N}$ and their ratio sets}, {DOI.10.13140/RG.2.2.28136.57608}

\bibitem{Salat}
T. \v{S}al\'{a}t, {On ratio sets of  natural numbers}, {\it Acta Arith.}, {\bf 15} (1969), 273-278.

\bibitem{Salat2}
T. \v{S}al\'{a}t, {Corrigendum to the paper “On ratio sets of sets of natural numbers}, {\it Acta Arith.}, {\bf 16} (1969/1970),
103.

\bibitem{Salat3}
T. \v{S}al\'{a}t, {Quotientbasen und (R)-dichte Mengen. (German)}, {\it Acta Arith.}, {\bf 19} (1971), 63-78.

\bibitem{Sanna}
C. Sanna, {The quotient set of $k$-generalised Fibonacci numbers is dense in $\mathbb{Q}_{p}$}, {\it Bull. Austral. Math. Soc.}, {\bf 96} (1) (2017), 24-29.


\bibitem{Sittinger}
B. D. Sittinger, {Quotients of primes in a quadratic number ring}, arXiv:1607.08319 (2016).

\bibitem{ps}
P. Starni, {Answers to two questions concerning quotients of primes}, {\it Amer. Math. Monthly}, {\bf 102} (4) (1995), 347-349.

\bibitem{strauch}
O. Strauch, {Distribution functions of ratio sequences, an expository paper}, {\it  Tatra Mt. Math. Publ.}, {\bf 64} (2015), 133-185.

\bibitem{Toth}
O. Strauch and J. T. T\'{o}th, {Asymptotic density of $A \subset \mathbb{N}$ and density of the ratio set}, {\it Acta Arith.}, {\bf 87} (1998), 67-78.

\bibitem{Toth2}
O. Strauch and J. T. T\'{o}th, {Corrigendum to Theorem 5 of the paper:  Asymptotic density of $A \subset \mathbb{N}$ and
density of the ratio set $R(A)$}, {\it Acta Arith.}, {\bf 103} (2002), 191- 200.

\end{thebibliography}
\end{document}